\documentclass[letterpaper, 10 pt, conference]{ieeeconf}  

\IEEEoverridecommandlockouts                              

\usepackage{graphicx}%
\usepackage{multirow}%
\usepackage{amsmath,amssymb,amsfonts}%
\usepackage{amsthm}%
\let\labelindent\relax
\usepackage{enumitem}
\usepackage{mathrsfs}%
\usepackage{xcolor}%
\usepackage{textcomp}%
\usepackage{manyfoot}%
\usepackage{booktabs}%
\usepackage{algorithm}%
\usepackage{algorithmicx}%
\usepackage{algpseudocode}%
\usepackage{listings}%
\usepackage{epstopdf}
\usepackage{mathtools}
\usepackage{subcaption} 
\usepackage{amsthm}
\usepackage{changes}
\usepackage{subcaption}
\usepackage{graphicx}
\usepackage{subcaption}
\usepackage{caption}
\newtheorem{assumption}{Assumption}  
\newtheorem{theorem}{Theorem}  
\newtheorem{lemma}{Lemma}  
\newtheorem{rem}{Remark}  
\newtheorem{cor}{Corollary}  

\newcommand{\R}{{\mathbb{R}}}

\newcommand{\BEQ}{\begin{equation}}
\newcommand{\EEQ}{\end{equation}}

\newcommand{\mb}{\mathbf}

\title{\LARGE \bf
An Accelerated Proximal Bundle Method with Momentum}

\author{Zhuoqing Zheng$^1$, Junshan Yin$^1$, Shaofu Yang$^2$, and Xuyang Wu$^1$ 
\thanks{$^1$Z. Zheng, J. Yin, and X. Wu are with the School of Automation and Intelligent Manufacturing, Southern University of Science and Technology, Shenzhen 518055, China, and the State Key Laboratory of Autonomous Intelligent Unmanned Systems, Beijing 100081, China.  Email: {\tt\footnotesize 12433033@mail.sustech.edu.cn; 12431375@mail.sustech.edu.cn; wuxy6@sustech.edu.cn}}
\thanks{$^2$S. Yang is with School of Computer Science and Engineering, Southeast University, Nanjing 21189, China. Email: {\tt\footnotesize  sfyang@seu.edu.cn}}
\thanks{This work is supported in part by the Guangdong Provincial Key Laboratory of Fully Actuated System Control Theory and Technology under grant No. 2024B1212010002, in part by the Shenzhen Science and Technology Program under grant No. JCYJ20241202125309014, in part by the Shenzhen Science and Technology Program under grant No. KQTD20221101093557010, and in part by the Guangdong Basic and Applied Basic Research Foundation under Grant No. 2026A1515012017.}
}

\begin{document}

\maketitle
\thispagestyle{empty}
\pagestyle{empty}

\begin{abstract}
    Proximal bundle methods (PBM) are a powerful class of algorithms for convex optimization. Compared to gradient descent, PBM constructs more accurate surrogate models that incorporate gradients and function values from multiple past iterations, which leads to faster and more robust convergence. However, for smooth convex problems, PBM only achieves an $O(1/k)$ convergence rate, which is inferior to the optimal $O(1/k^2)$ rate. To bridge this gap, we propose an accelerated proximal bundle method (APBM) that integrates Nesterov's momentum into PBM. We prove that under standard assumptions, APBM achieves the optimal $O(1/k^2)$ convergence rate. Numerical experiments demonstrate the effectiveness of the proposed APBM.
\end{abstract}

\section{Introduction}
This article addresses the unconstrained convex optimization problem:
\begin{equation}\label{problem}
       \underset{x \in \mathbb{R}^n}{\operatorname{minimize}} ~~ f(x),
\end{equation}
where $f : \mathbb{R}^n \rightarrow \mathbb{R}$ is convex and differentiable. This 
problem has found numerous applications in machine learning \cite{wright2022optimization}, control system \cite{agrawal2020learning}, and signal processing \cite{luo2006introduction}.

A large number of algorithms have been proposed to solve problem \eqref{problem}, among which a typical method is gradient descent (GD) \cite{boyd2004convex}, which minimizes the proximal linear model (first-order Taylor expansion plus a proximal term) of the objective function $f$ at each iteration.
However, the proximal-linear model can be a crude
approximation of the objective function, which further leads to slow convergence. To enhance GD, the proximal bundle method (PBM) \cite{lemarechal1975extension,diaz2023optimal,Nesterov2021memory} incorporates a proximal bundle model into the update. Compared to the proximal linear model, the proximal bundle model incorporates historical objective values and gradients to achieve a higher approximation accuracy \cite{Nesterov2021memory}.  Compared to GD, PBM not only converges fast but also exhibits extraordinary robustness in the step-size \cite{lemarechal1993convex}. For this reason, a growing body of works extends PBM to solve optimization problems with different structures  \cite{zhu2025historical,cederberg2025an,hare2010redistributed,de2019proximal}. 

Another effective approach for improving GD is to incorporate Nesterov's momentum scheme into the update, resulting in the accelerated gradient descent (AGD) \cite{nesterov1983method,Beck09FISTA,d2021acceleration}. Compared with GD, AGD employs the gradient descent step at the specific linear combination (controlled by momentum coefficient) of the previous two points, but rather at the most recent point. AGD achieves the $O(1/k^2)$ convergence rate for convex smooth optimization, which is also the optimal rate that can attained by gradient-based methods under this setting \cite[Section 2.1.2]{nesterov2018lectures}. Inspired by the extraordinary performance of momentum scheme, \cite{Beck09FISTA} extends it to solve the problem with composite structure, \cite{xu2017accelerated} applies it to primal-dual algorithms, and \cite{huang2025accelerated} adapts it to the distributed optimization.

While PBM exhibits excellent convergence performance, its convergence rate, for convex smooth problems, is typically of $O(1/k)$ \cite{Nesterov2021memory,diaz2023optimal}, rather than the optimal rate $O(1/k^2)$. On the other hand, AGD achieves the optimal rate but its update still relies on a gradient descent step that can potentially lead to slow convergence and lack of robustness. A natural idea is to combine PBM with AGD, taking advantage of both to obtain a better algorithm. Although this idea is also explored in the existing works \cite{liao2025accelerated,fersztand2025acceleration}, they both include an internal loop that complicates the algorithm.

This article incorporates the momentum scheme into the PBM, resulting in an accelerated proximal bundle method (APBM) that can attain the accelerated $O(1/k^2)$ convergence rate while maintaining the robustness of PBM. The main contributions of this paper are as follows:
\begin{enumerate}[label=\arabic*)]
    \item We propose an accelerated proximal bundle method that combines PBM with the momentum scheme. Compared to \cite{liao2025accelerated,fersztand2025acceleration}, our method does not include any inner loop, which yields a simpler form.
    \item We provide a convergence rate of $O(1/k^2)$ for our algorithm under standard assumptions.
    \item We demonstrate the practical advantages of our method by numerical experiments. 
\end{enumerate}

The remaining part of this paper is organized as follows: Section \ref{sec:algorithms} develops the algorithm and Section \ref{sec:convergence} analyses the convergence. Section \ref{sec:numrical} illustrates the practical advantages of our method through numerical experiments, and Section \ref{sec:conclusion} concludes the paper.

\textbf{Notations and definitions.} We denote by $\langle \cdot,\cdot \rangle$ the Euclidean inner product and use $\|\cdot\|$ to represent the Euclidean norm for vectors and the spectral norm for matrices. For any differentiable $f:\mathbb{R}^n \rightarrow \mathbb{R}$, $\nabla  f(x) \in \mathbb{R}^n$ represents its gradient at $x$. We use $\mb{1}$ to denote the all-one vector of proper dimensions. For a differentiable function $f$, we say it is $L$-smooth for some $L>0$ if
\[\|\nabla f(x)-\nabla f(y)\| \leq L\|x-y\|\quad \forall x,y \in \mathbb{R}^n,\]
and it is $\mu$-strongly convex for some $\mu>0$ if 
\[\langle \nabla f(x) - \nabla f(y), x-y\rangle \geq \mu \|x-y\|^2\quad \forall x,y \in \mathbb{R}^n.\] 

\section{Algorithm}\label{sec:algorithms}
This section is organized into three parts. First, we develop the algorithm, which involves a surrogate model $\hat{f}^k$ and a subproblem at each iteration. Next, we discuss options of $\hat{f}^k$. Finally, we show that the resulting subproblems can be solved efficiently via dual approaches..
\subsection{Algorithm development}
The algorithm is developed based on the proximal bundle method (PBM) and Nesterov's momentum scheme.

    {\bf PBM}: At each iteration $k\ge 0$, it updates as \cite{diaz2023optimal, Nesterov2021memory}:
\begin{equation*}
    x^{k+1} = \arg\min_x ~ \hat{f}^k(x) + \frac{1}{2\gamma}\|x-x^k\|^2,
\end{equation*}
where $\hat{f}^k$ is a minorant of $f$ satisfying $\hat{f}^k(x)\le f(x)$ for all $x\in\R^n$ and $\gamma>0$ is the step-size. A typical example of $\hat{f}^k$ is the cutting-plane model \cite{kelley1960cutting}: 
\begin{equation*}
    \hat{f}^k(x)  = \max_{t\in S^k} \{ f(x^t) + \langle \nabla f(x^t),x-x^t \rangle\},
\end{equation*}
where $S^k \subseteq \{ 1,\dots,k\}$ is an index set that contains $k$. When $S^k = \{k\}$, the proximal bundle model reduces to the first-order Taylor expansion of $f$ and PBM becomes gradient descent (GD). The use of multiple cutting planes yields a higher approximation accuracy of $\hat{f}^k$ on $f$ compared to the first-order Taylor expansion: since $f(x) \geq \hat{f}^k(x) \geq f(x^k) + \langle \nabla f(x^k),x-x^k\rangle$, we have 
\[ f(x) - \hat{f}^k(x) \leq f(x) - (f(x^k) + \langle \nabla f(x^k),x-x^k\rangle), \]
which is also clear from Fig. \ref{fig:surrogate functions}~(b). The higher approximation accuracy yields not only faster convergence but also stronger robustness in the step-size \cite{lemarechal1993convex}.

{\bf Nesterov’s momentum scheme}: It is a powerful technique in enhancing the performance of first-order methods, and an example is accelerated gradient descent (AGD) \cite{nesterov1983method,Beck09FISTA,d2021acceleration}, which achieves the optimal $O(1/k^2)$ convergence rate for smooth convex optimization. By assuming $L$-smooth $f$ for some $L>0$, the AGD in \cite{Beck09FISTA} updates as: Set $y^1=x^0$ and $t_1=1$. For each $k\ge 1$,
\begin{align}
    &x^k = y^k - \frac{1}{L}\nabla f(y^k),\label{eq:AGD-x} \\
    &t_{k+1} = \frac{1+ \sqrt{1+4t_k^2}}{2}, \label{eq:AGD-t} \\
    &y^{k+1} = x^k  + \frac{t_k-1}{t_{k+1}}(x^k-x^{k-1}).\label{eq:AGD-y}
\end{align}
The algorithm maintains three sequences: the iterative sequence $\{x^k\}$, coefficient sequence $\{t^k\}$, and the extrapolated sequence $\{y^k\}$. Compared with GD, AGD performs a gradient descent step at the extrapolated point $y^k$ rather than $x^k$. This extrapolation sequence $\{y^k\}$ is generated by the linear combination of the two most recent iterations $x^k$ and $x^{k-1}$, with the coefficient $\{ t_k\}$ controlling the momentum strength.

Both the proximal bundle model and Nesterov's momentum scheme can significantly improve the performance of GD. Therefore, a natural idea is to combine them for better performance. We keep the updates of the extrapolated sequence $\{y^k\}$ and the coefficient sequence $\{ t_k\}$ in AGD \eqref{eq:AGD-x}--\eqref{eq:AGD-y}, and incorporate the proximal bundle step into the update of $x^k$, leading to
\begin{equation}\label{eq:APBM-x}
    x^k = \arg\min_x ~\hat{f}^k(x)+ \frac{L}{2}\|x-y^k\|^2.
\end{equation}
We refer to the algorithm with \eqref{eq:APBM-x}, \eqref{eq:AGD-t}, and \eqref{eq:AGD-y} as the Accelerated Proximal Bundle Method (APBM), where a detailed implementation is described in Algorithm \ref{algo:APBM}. 

\begin{rem}
The works \cite{liao2025accelerated,fersztand2025acceleration} also incorporate Nesterov’s acceleration into PBM. However, both methods involve an inner loop that iterates until a prescribed condition is satisfied, which increases algorithmic complexity. In contrast, our method eliminates the need for such inner loops and admits a simpler structure.
\end{rem}

\begin{figure*}[t]
\centering
    \begin{subfigure}[t]{0.26\linewidth}
        \includegraphics[width=\linewidth]{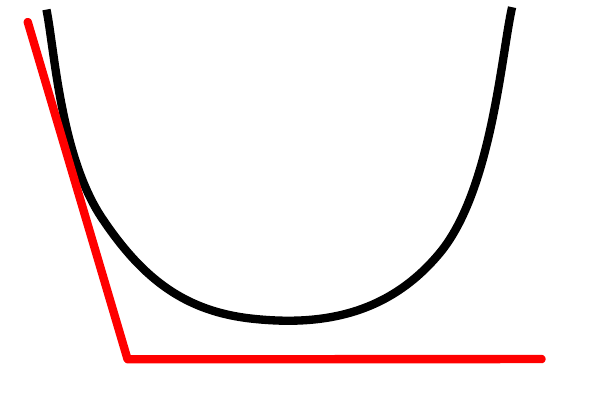}
        \caption{Polyak}
    \end{subfigure} 
    \hfill
    \begin{subfigure}[t]{0.26\linewidth}
        \includegraphics[width=\linewidth]{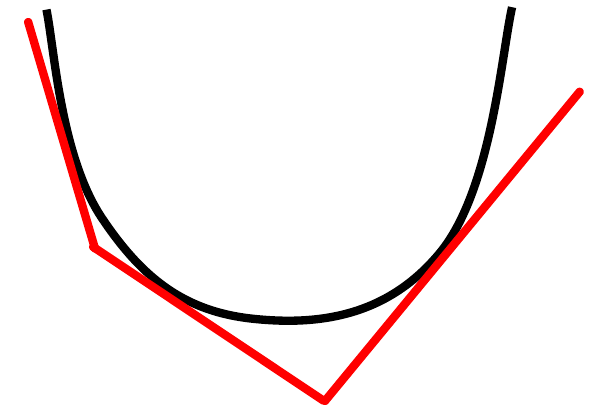}
        \caption{Cutting-plane}
    \end{subfigure} 
    \hfill
    \begin{subfigure}[t]{0.26\linewidth}
        \includegraphics[width=\linewidth]{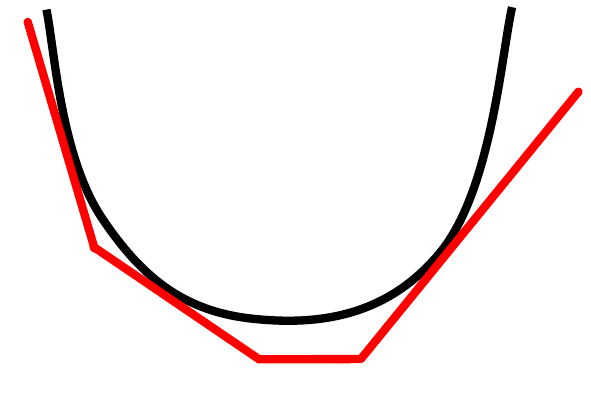}
        \caption{Polyak cutting-plane}
    \end{subfigure} 
    \caption{Surrogate functions in the Polyak model \eqref{eq:pol_model}, cutting-plane model \eqref{eq:cpm}, and the Polyak cutting-plane model \eqref{eq:pcpm}}
    \label{fig:surrogate functions}
\end{figure*}

 \begin{algorithm}[t]
\caption{Accelerated Proximal Bundle Method (APBM)}\label{algo:APBM}
\begin{algorithmic}[1]
\State \textbf{Initialization:} $x^0 \in \mathbb{R}^n,~y^1 = x^0,~t_1=1$. 
\For{$k = 1, 2, \ldots$}
\vspace{0.2cm}
\State $x^k = \arg\min_x ~\hat{f}^k(x)+ \frac{L}{2}\|x-y^k\|^2$.
\vspace{0.2cm}
\State $t_{k+1} = \frac{1+ \sqrt{1+4t_k^2}}{2}$.
\vspace{0.2cm}
\State $ y^{k+1} = x^k  + \frac{t_k-1}{t_{k+1}}(x^k-x^{k-1})$.
\vspace{0.2cm}
\EndFor
\end{algorithmic}
\end{algorithm}

\subsection{Candidates of the model $\hat{f}^k$}\label{ssec:candidates}

The performance of APBM relies on the selection of the bundle model $\hat{f}^k$ and, in general, we prefer models that yield a high approximation accuracy on $f$ and a low computational cost of solving \eqref{eq:APBM-x}. In this subsection, we introduce several candidates of $\hat{f}^k$ satisfying the following assumption, which requires $\hat{f}^k$ to be a convex minorant of $f$ and a majorant of the cutting plane at $y^k$ and will be used in Section \ref{sec:convergence} for convergence analysis.
\begin{assumption}\label{asm:pri_model}
    The model $\hat{f}^k$ satisfies
    \begin{enumerate}[label=(\alph*)]
    \item $\hat{f}^k$ is convex;\label{primal model assumption a}
    \item $\hat {f}^k(x) \geq f(y^k) + \langle \nabla f(y^k),x-y^k\rangle$ for all $x\in\R^n$;\label{primal model assumption b}
    \item $\hat{f}^k(x)\leq f(x)$ for all $x\in\R^n$.\label{primal model assumption c}
    \end{enumerate}
\end{assumption}

Below, we provide several candidates of $\hat{f}^k$ that satisfy Assumption \ref{asm:pri_model}, where three of them are visualized in Fig. \ref{fig:surrogate functions}. 
    \begin{itemize}[leftmargin=*]
        \item \textbf{Polyak model}: This model originates from the Polyak step-size \cite{polyak1987introduction} for gradient descent and takes the form of
        \begin{equation}\label{eq:pol_model}
            \hat{f}^k(x) = \max\{f(y^k)+\langle\nabla f(y^k), x-y^k\rangle, \ell_f\}
        \end{equation}
    where $\ell_f$ is a known lower bound of $f$. This model and its variants are shown to be particularly effective in stochastic optimization \cite{wang2023generalized}.
    \item \textbf{Cutting-plane model}: It takes the maximum of historical cutting planes \cite{kelley1960cutting}:
        \begin{equation}\label{eq:cpm}
            \hat{f}^k (x) = \max_{t\in S^k}~f(y^t)+\langle \nabla f(y^t),x-y^t\rangle
        \end{equation}
        where $S^k\subseteq \{0,1,\ldots,k\}$ is a subset of historical iteration indexes containing $k$. The model is adopted in the cutting-plane method \cite{kelley1960cutting} and is also typical in the proximal bundle method \cite{cederberg2025an}.
    \item \textbf{Polyak cutting-plane model}: It takes the maximum of the Polyak model and the cutting-plane model:
        \begin{equation}\label{eq:pcpm}
            \hat{f}^k (x) = \max\{\ell_f, \max_{t\in S^k}~f(y^t)+\langle \nabla f(y^t),x-y^t\rangle\}
        \end{equation}
        where all the parameters are introduced below \eqref{eq:pol_model} and \eqref{eq:cpm}.
    \item \textbf{Two-cut model}: This model is defined in an iterative way. It sets $\hat{f}^1(x) = f(y^1)+\langle\nabla f(y^1), x-y^1\rangle$ and takes the maximum of the cutting planes of $f$ at $y^k$ and $\hat{f}^{k-1}$ at $x^{k-1}$ for $k\geq 2$:
        \begin{equation}\label{eq:two-cut}
        \begin{split}
            \hat{f}^k(x) = &\max\{\hat{f}^{k-1}(x^{k-1})+\langle \hat{g}^{k-1}, x-x^{k-1}\rangle, \\&f(y^k)+\langle \nabla f(y^k),x-y^k\rangle\}
        \end{split}
        \end{equation}
        where $\hat{g}^{k-1} = L(y^{k-1}-x^{k-1})\in \partial \hat{f}^{k-1}(x^{k-1})$.
    \end{itemize}
    The models \eqref{eq:pol_model}--\eqref{eq:two-cut} incorporate historical information or lower bounds to approximate $f$. Compared to the first-order Taylor expansion, they can yield a higher approximation accuracy (see Fig. \ref{fig:surrogate functions}).
    
    The following lemma shows that the models \eqref{eq:pol_model}--\eqref{eq:two-cut} satisfy Assumption \ref{asm:pri_model}.

    \begin{lemma}\label{lemma:mod_sat_asm}
        Suppose that $f$ is convex and differentiable. Then, the four models \eqref{eq:pol_model}--\eqref{eq:two-cut} satisfy Assumption \ref{asm:pri_model}.
    \end{lemma}
    \begin{proof}
        See Appendix \ref{appen:mod_sat_asm}.
    \end{proof}

\subsection{Solving the subproblem \eqref{eq:APBM-x}}
The efficiency of the algorithm heavily relies on that of
solving the subproblem \eqref{eq:APBM-x}. In this subsection, we show that for piece-wise linear $\hat{f}^k$, such as the four options in Section \ref{ssec:candidates}, the subproblem \eqref{eq:APBM-x} can be solved quickly.

For simplicity, we rewrite the update \eqref{eq:APBM-x} with a piece-wise linear $\hat{f}^k$ as
\begin{equation}\label{eq:subproblem_primal}
    \underset{x}{\operatorname{minimize}} ~\max_{i\in\{1,\dots,M\}} \{a_i^Tx+b_i\} + \frac{L}{2}\|x-y^k\|^2,
\end{equation}
where $M$ is the number of affine functions in $\hat{f}^k$ and $a_i^T + b_i$ is the $i$-th affine function. Taking the cutting-plane model \eqref{eq:cpm} as an example,
\begin{align*}
     a_i = \nabla f(y^{t_i}),~~~~~ b_i = f(y^{t_i}) -\langle \nabla f(y^{t_i}),y^{t_i} \rangle,
\end{align*}
where $t_i$ is the $i$th element in $S^k$. By letting
\begin{equation}\label{eq:tilde_A}
    A=(a_1, \ldots, a_M)^T\in\R^{M\times n},\quad  b = (b_1;\ldots;b_M)\in\R^M,
\end{equation}
problem \eqref{eq:subproblem_primal} can be equivalently transformed into
\begin{equation}\label{eq:QP}
    \begin{split}
        \underset{x,\theta}{\operatorname{minimize}}~~&~\theta+\frac{L}{2}\|x-y^k\|^2\\
        \operatorname{subject~to}~ &~Ax+b \le \theta\mathbf{1}.
    \end{split}
\end{equation}
If the dimension $n$ of $x$ is small, then problem \eqref{eq:QP} can be easily solved by primal methods such as interior-point methods \cite{goulart2024clarabel}. 

In case the dimension $n$ of $x$ is large, dual methods become more preferable for solving problem \eqref{eq:QP}. By \cite[Lemma 2.4.1]{lemarechal1993convex}, the dual problem of \eqref{eq:QP} is 
\begin{equation}\label{eq:subproblem_dual}
    \begin{split}
        &\underset{\lambda\in\mathbb{R}^M}{\operatorname{maximize}} ~~~~~~ q(\lambda) \\
        &\operatorname{subject ~to}   ~~~~\lambda  \geq 0, ~\mathbf{1}^T\lambda = 1,
    \end{split}
\end{equation}
where $q(\lambda) = -\frac{1}{2L}\|A^T\lambda\|^2+\langle \lambda,Ay^k+b\rangle$. Problem \eqref{eq:subproblem_dual} is a QP with dimension $M$, which is typically small (e.g., $5$ or $10$) in practical implementations. Therefore, compared to directly solving the primal problem \eqref{eq:subproblem_primal}, solving the dual problem \eqref{eq:subproblem_dual} can admit a much low computational cost especially when $M\ll n$. Moreover, problem \eqref{eq:subproblem_dual} can be solved by gradient-based approaches, such as projected gradient descent \cite{boyd2004convex}, because both the gradient computation and the projection operation is simple. To see this, note that by letting $x_\lambda = y^k-\frac{1}{L}A^T\lambda$, we have
\begin{equation*}
 \nabla q(\lambda) = Ax_\lambda+b.
\end{equation*}
Moreover, the projection onto the simplex constraint set can be executed efficiently ($O(M)$) \cite{condat2016fast}.  Once an optimum $\lambda^{\operatorname{opt}}$ of \eqref{eq:subproblem_dual} is solved, the optimum of \eqref{eq:QP} can be recovered by
\begin{equation*}
    x^{\operatorname{opt}} =  y^k-\frac{1}{L}A^T\lambda^{\operatorname{opt}}.
\end{equation*}
When $n=10,000$ and $M=10$, applying projected gradient descent to solve \eqref{eq:subproblem_dual} only takes  $\approx 0.08$ second to reach the $x^{\operatorname{opt}}$ of \eqref{eq:QP} with $10^{-10}$ accuracy on a PC with the AMD Ryzen 7 CPU.

\section{Convergence Analysis}\label{sec:convergence}
In this section, we analyse the convergence of APBM. To this end, we impose an assumption on problem \eqref{problem}.

\begin{assumption}\label{asm:prob}
The following holds:
     \begin{enumerate}[label=(\alph*)]
     \item The function $f: \mathbb{R}^n \rightarrow \mathbb{R}$ is convex and $L$-smooth for some $L>0$.
     \item Problem \eqref{problem} has at least one optimal solution $x^\star$.
    \end{enumerate}
\end{assumption}

Assumption \ref{asm:prob} is standard in the analysis of gradient-based methods \cite{boyd2004convex,Beck09FISTA,liao2025accelerated}. Under this problem setting, GD and PBM attain an $O(1/k)$ convergence rate \cite{boyd2004convex,Nesterov2021memory,diaz2023optimal}, while AGD could achieve the $O(1/k^2)$ convergence rate.

The following theorem demonstrates that the function value error $f(x^k)-f(x^\star)$ can be upper bounded by $t_k$ that generated by \eqref{eq:AGD-t}. 

\begin{theorem}\label{theo:sublinear}
     Suppose that Assumptions \ref{asm:pri_model} -- \ref{asm:prob}  hold. Let $\{x^k\}$ be generated by APBM (Algorithm \ref{algo:APBM}). Then, for any $k\geq 1$,
     \begin{equation}\label{eq:result-tk}
         f(x^{k}) -f(x^\star) \leq \frac{L\|x^0-x^\star\|^2}{2t_{k}^2}.
     \end{equation}
\end{theorem}

\begin{proof}
    See Appendix \ref{appe:theo1}.
\end{proof}

To establish the relationship between the function value error and the number of iterations, the following technical lemma provides a crucial lower bound of $t_k$, which implies that $t_k$ grows linearly with the iteration number. 
\begin{lemma}{\cite[Lemma 4.3]{Beck09FISTA}}\label{lem:tk}
    Let $\{t_k\}$ be generated by \eqref{eq:AGD-t} with $t_1 =1$. Then, for any $ k\geq1$, 
    \begin{equation}\label{eq:tk_geq_k+1}
        t_k \geq \frac{k+1}{2}.
    \end{equation}
\end{lemma}

Based on the Theorem \ref{theo:sublinear} and Lemma \ref{lem:tk}, below we provide the final result.

\begin{cor}\label{cor:APBM}
    Suppose all the conditions in Theorem \ref{theo:sublinear} hold. Then, for any $k \geq 0$,
    \begin{equation}\label{eq:result}
         f(x^{k}) -f(x^\star) \leq \frac{2L\|x^0-x^\star\|^2}{(k+1)^2}.
     \end{equation}
\end{cor}

\begin{proof}
    See Appendix \ref{appen:cor}.
\end{proof}

APBM achieves the optimal $O(1/k^2)$ convergence rate for convex and smooth optimization \cite[Section 2.1.2]{nesterov2018lectures}. APBM generalizes AGD and the convergence result is the same order as AGD \cite{nesterov1983method,Beck09FISTA}. However, due to the use of more accurate models, APBM can yield much faster convergence and stronger robustness, which will be shown numerically in Section \ref{sec:numrical}.

\begin{figure*}[ht]
    \centering
    \begin{subfigure}{0.45\linewidth}
    \includegraphics[width=\linewidth]{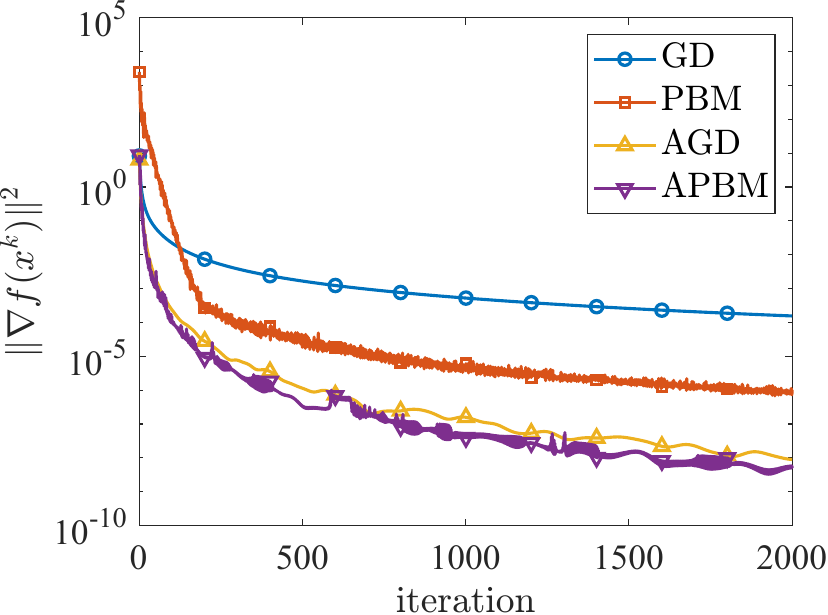}
    \caption*{(a) Convergence of GD, PBM ($m=15$), AGD, and APBM ($m=15$)}
    \end{subfigure}
    \hfill
   \begin{subfigure}{0.43\linewidth}
          \includegraphics[width=\linewidth]{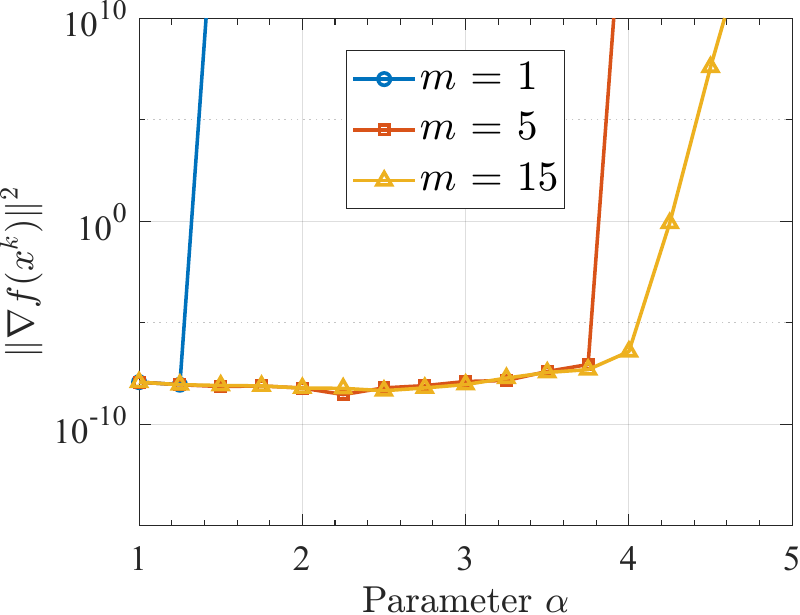}
    \caption*{(b) Final optimality residual of APBM after 2000 iterations with the step-size $\gamma=\alpha/L$ where $L=\|E\|^2/N$}
   \end{subfigure}
    \caption{Convergence performance in solving the least squares problem \eqref{problem:LS}}
    \label{fig:LS}
\end{figure*}

\begin{figure*}[ht]
    \centering
    \begin{subfigure}{0.45\linewidth}
    \includegraphics[width=\linewidth]{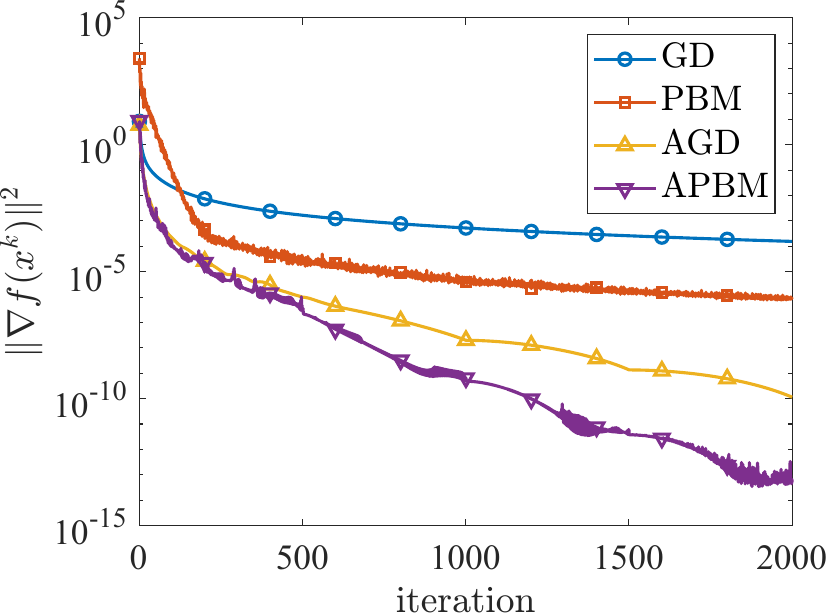}
    \caption*{(a) Convergence of GD, PBM ($m=15$), AGD, and APBM ($m=15$)}
    \end{subfigure}
    \hfill
   \begin{subfigure}{0.43\linewidth}
          \includegraphics[width=\linewidth]{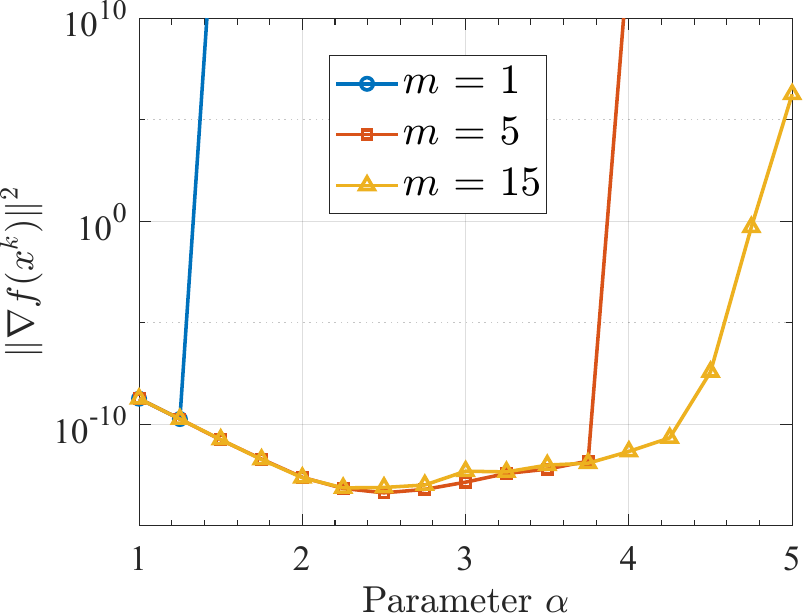}
    \caption*{(b) Final optimality residual of APBM after 2000 iterations with the step-size $\gamma=\alpha/L$, where $L=\|E\|^2/N$}
   \end{subfigure}
    \caption{Convergence performance (with restart scheme) in solving the least squares problem \eqref{problem:LS}}
    \label{fig:LS-restart}
\end{figure*}

\section{Numerical Example}\label{sec:numrical}
We demonstrate the performance of our method in solving the following least squares problem:
\begin{equation}\label{problem:LS}
    \underset{x\in\mathbb{R}^n}{\operatorname{minimize}}~~\frac{1}{2N}\|Ex-w\|^2,
\end{equation}
where $N=800$, $n=800$, and $E\in\mathbb{R}^{N\times n}$, $w\in \mathbb{R}^{N}$ are randomly generated.

We compare GD, PBM, AGD, and APBM in solving \eqref{problem:LS}. The experiment settings are as follows: Both PBM and APBM adopt the cutting-plane model \eqref{eq:cpm} with $S_k=[k-m+1,k]$. For the updates \eqref{eq:AGD-x} and \eqref{eq:APBM-x} in AGD and APBM, we rewrite them as
\begin{align*}
    x^{k+1} &= \operatorname{\arg\;\min}_x f(y^k)\!+\!\langle\nabla f(y^k),x\!-\!y^k\rangle \!+\!\frac{1}{2\gamma}\|x\!-\!y^k\|^2,\\
    x^{k+1} &= \operatorname{\arg\;\min}_x \hat{f}^k(x)+\frac{1}{2\gamma}\|x-y^k\|^2,
\end{align*}
respectively, where $\gamma>0$ is referred to as \emph{step-size} and the above updates reduce to \eqref{eq:AGD-x} and \eqref{eq:APBM-x} when $\gamma=1/L$.

We first compare the convergence speed where the step-size in all methods are fine-tuned for better performance and then test the robustness of the algorithms with respect to the step-size $\gamma$. The experimental results are plotted in Fig \ref{fig:LS}, which demonstrate that APBM takes the advantages of AGD and PBM:
\begin{enumerate}[label=\arabic*)]
    \item By utilizing the Nesterov's momentum scheme, APBM converges faster than PBM (at least 3x faster);
    \item Benefiting from the more accurate proximal bundle model, APBM has stronger robustness than AGD, which eases parameter selection. In particular, AGD is equivalent to APBM with $m=1$ and diverges when step-size $\geq 1.4/L$, whereas APBM with $m = 10$ allows for $4/L$.
\end{enumerate}

To further improve the performance, we incorporate a fixed-restart scheme \cite{o2015adaptive} into both AGD and APBM, where $t_k$ is set to $1$ every $500$ iterations. The results are displayed in Fig \ref{fig:LS-restart}. Compared with Fig.~\ref{fig:LS}, the restart scheme enhances the performance of both methods; however, the improvement is more substantial for APBM. This suggests that APBM benefits more from the restart scheme, leading to faster convergence than AGD. Moreover, the restart scheme preserves the strong robustness of APBM with respect to the step-size.

\section{Conclusion}\label{sec:conclusion}
We proposed an APBM that integrates Nesterov's momentum scheme into the classical PBM. The proposed algorithm achieves the optimal $O(1/k^2)$ convergence rate for convex and smooth problems while preserving the robustness and fast convergence properties of PBM. We provided the theoretical convergence guarantee under standard assumptions and demonstrated the fast convergence and robustness through numerical experiments. We consider further accelerating APBM by introducing additional mechanisms such as restart schemes and establishing their theoretical guarantees, as our future work.

\begin{appendix}

\subsection{Proof of Lemma \ref{lemma:mod_sat_asm}}\label{appen:mod_sat_asm}
Since $\hat{f}^k$ in \eqref{eq:pol_model}--\eqref{eq:two-cut} are the maximum of affine functions, they are convex and satisfy Assumption \ref{asm:pri_model} (a). Assumption \ref{asm:pri_model} (b) is straightforward to see from the forms of $\hat{f}^k$ in \eqref{eq:pol_model}--\eqref{eq:two-cut}.  
    
    To show Assumption \ref{asm:pri_model} (c), note that since $f$ is convex,
    \[ f(x) \geq  f(y^t) + \langle \nabla f(y^t),x-y^t\rangle, \quad\forall  t \in S^k,\]
    i.e., $f(y^t) + \langle \nabla f(y^t),x-y^t\rangle, t \in S^k$ are minorants of $f$. Moreover, $\ell_f\le \min_x f(x)$. Therefore, the models \eqref{eq:pol_model}--\eqref{eq:pcpm} take the maximum of minorants of $f$, which yields $\hat{f}^k(x) \le f(x)$. For $\hat{f}^k$ in \eqref{eq:two-cut}, we show $\hat{f}^k \leq f$ by induction. By the convexity of $f$ and $y^1 =x^0$, the initial model \[\hat{f}^1(x) = f(x^0) + \langle \nabla f(x^0),x-x^0\rangle \leq f(x).\] Assume that for some $k \geq 1$, we have $\hat{f}^k(x) \leq f(x)$. By $f(x)\ge \hat{f}^k(x)$ and the convexity of $\hat{f}^k$, we have that for any $\hat{g}^{k}\in\partial\hat{f}^k(x^{k})$,
    \[f(x)\ge \hat{f}^k(x) \geq \hat{f}^k(x^{k}) + \langle \hat{g}^{k},x-x^{k}\rangle,\]
    which, together with the convexity of $f$, yields
    \begin{equation*}
    \begin{split}
        \hat{f}^{k+1}(x)=&\max\{\hat{f}^k(x^{k})+\langle \hat{g}^{k},x-x^k\rangle,\\ &f(y^k)+\langle \nabla f(y^k),x-y^k\rangle\} \\ \leq& 
        f(x).
    \end{split}
    \end{equation*}
    Concluding all the above, Assumption \ref{asm:pri_model} (c) holds for all $k\geq 1$.


\subsection{Proof of Theorem \ref{theo:sublinear}}\label{appe:theo1}
For any $k\ge 0$, define \[\phi^{k+1}(x) = \hat{f}^{k+1}(x) + \frac{L}{2}\|x-y^{k+1}\|^2.\] By Assumption \ref{asm:pri_model} (a) and the $L$-strong convexity of $\frac{L}{2}\|x-y^{k+1}\|^2$, $\phi^{k+1}(x)$ is $L$-strongly convex. This together with $x^{k+1} = \arg\min_x \phi^{k+1}(x)$ yield
\begin{equation}\label{eq:Lk}
    \phi^{k+1}(x^{k+1}) - \phi^{k+1}(x) \leq -\frac{L}{2}\|x^{k+1}-x\|^2, ~~\forall x \in \mathbb{R}^n.
\end{equation}
By Assumption \ref{asm:pri_model} (b) and the smoothness of $f$, we have
\begin{equation}\label{eq:Lk-1}
    \begin{split}
        &\phi^{k+1}(x^{k+1})\\ =& \hat{f}^{k+1}(x^{k+1}) + \frac{L}{2}\|x^{k+1}-y^{k+1}\|^2 \\
        \geq& f(y^{k+1})+\langle \nabla f(y^{k+1}),x^{k+1}-y^{k+1}\rangle \\ &+ \frac{L}{2}\|x^{k+1}-y^{k+1}\|^2 \\
        \geq& f(x^{k+1}).
    \end{split}
\end{equation}
By Assumption \ref{asm:pri_model} (c), 
\begin{equation}\label{eq:Lk-2}
    \begin{split}
        \phi^{k+1}(x) &= \hat{f}^{k+1}(x) + \frac{L}{2}\|x-y^{k+1}\|^2 \\
    &\leq f(x) + \frac{L}{2}\|x-y^{k+1}\|^2.
    \end{split}
\end{equation}
Substituting \eqref{eq:Lk-1} and \eqref{eq:Lk-2} into \eqref{eq:Lk}, we have
\begin{equation}\label{eq:A}
    f(x^{k+1})-f(x) \leq \frac{L}{2}\|x-y^{k+1}\|^2 - \frac{L}{2}\|x-x^{k+1}\|^2.
\end{equation}
Moreover,
\begin{equation*}
\begin{split}
    & \frac{L}{2}\|x-y^{k+1}\|^2 - \frac{L}{2}\|x-x^{k+1}\|^2\\
    =&\frac{L}{2}\|x^{k+1}-y^{k+1}\|^2+L\langle x-x^{k+1},x^{k+1}-y^{k+1} \rangle.
\end{split}
\end{equation*}
Therefore,
\begin{equation*}
\begin{split}
      &f(x^{k+1})-f(x) \\\leq& \frac{L}{2}\|x^{k+1}-y^{k+1}\|^2 + L\langle x-x^{k+1},x^{k+1}-y^{k+1} \rangle.
\end{split}  
\end{equation*}
Letting $x = x^k$ and $x = x^\star$ respectively, we obtain
\begin{equation}\label{eq:111}
    \begin{split}
        &f(x^{k+1})-f(x^k)  \\ \leq& \frac{L}{2}\|x^{k+1}-y^{k+1}\|^2 + L\langle x^k-x^{k+1},x^{k+1}-y^{k+1} \rangle  ,
    \end{split}
\end{equation}
\begin{equation}\label{eq:222}
\begin{split}
    &f(x^{k+1})-f(x^\star)  \\ \leq& \frac{L}{2}\|x^{k+1}-y^{k+1}\|^2 + L \langle x^\star-x^{k+1},x^{k+1}-y^{k+1}\rangle .
\end{split}
\end{equation}
For conciseness, let $v^k = f(x^k)-f(x^\star)$. To get a relationship between $v^{k+1}$ and $v^k$, we multiply \eqref{eq:111} by $(t_{k+1}-1)$ and add the resulting equation to \eqref{eq:222}, which yields
\begin{equation*}
    \begin{split}
        &t_{k+1}v^{k+1} - (t_{k+1}-1)v^k  \\ \leq&   L \langle (t_{k+1}-1)(x^k-x^{k+1})+x^\star-x^{k+1},x^{k+1}-y^{k+1}\rangle \\
        &+ t_{k+1}\frac{L}{2}\|x^{k+1}-y^{k+1}\|^2.
    \end{split}
\end{equation*}
Multiplying both sides of the above inequality by $t_{k+1}$ and using $t_{k}^2 = t_{k+1}(t_{k+1}-1)$, we obtain
\begin{equation*}
    \begin{split}
        &t_{k+1}^2v^{k+1} - t_k^2v^k  \leq \frac{L}{2}\|\underbrace{t_{k+1}(x^{k+1}-y^{k+1})}_{\mathbf{b}}\|^2 \\+& L\langle \underbrace{(t_{k+1}-1)x^k-t_{k+1}x^{k+1}+x^\star}_{\mathbf{a}},\underbrace{t_{k+1}(x^{k+1}-y^{k+1}}_{\mathbf{b}})
\rangle.
    \end{split}
\end{equation*}
By $\|\mathbf{a}+\mathbf{b}\|^2 - \|\mathbf{a}\|^2 = 2\langle \mathbf{a},\mathbf{b}\rangle+\|\mathbf{b}\|^2$, it follows that
\begin{equation}\label{eq:NNN}
    \begin{split}
        &t_{k+1}^2v^{k+1} - t_k^2v^k \leq \frac{L}{2}\|x^\star+(t_{k+1}-1)x^k-t_{k+1}y^{k+1}\|^2 \\ &- \frac{L}{2}\|x^\star+(t_{k+1}-1)x^k-t_{k+1}x^{k+1}\|^2.
    \end{split}
\end{equation}
By \eqref{eq:AGD-y}, we have $t_{k+1}y^{k+1} = t_{k+1}x^k + (t_k-1)(x^k-x^{k-1})$, substituting which into \eqref{eq:NNN} gives
\begin{equation*}
    \begin{split}
        &t_{k+1}^2v^{k+1} - t_k^2v^k \leq \frac{L}{2}\|x^\star+(t_{k}-1)x^{k-1}-t_{k}x^{k}\|^2  \\ &- \frac{L}{2}\|x^\star+(t_{k+1}-1)x^k-t_{k+1}x^{k+1}\|^2.
    \end{split}
\end{equation*}
Using $t_1 =1$ and applying telescoping cancellation on the above equation yields that for all $k \geq 1$,
\begin{align*}
    t_{k}^2v^{k}-v^1 &\leq \frac{L}{2}\|x^\star+(t_{1}-1)x^{0}-t_{1}x^{1}\|^2 \\
    & = \frac{L}{2}\|x^\star-x^1\|^2.
\end{align*}
By $v^k = f(x^k)-f(x^\star)$, \eqref{eq:A}, and $y^1=x^0$, we have that for all $k \geq 1$
\begin{equation*}
    \begin{split}
    &t_{k}^2(f(x^{k})-f(x^\star)) \\ \leq& f(x^1) -f(x^\star) + \frac{L}{2}\|x^\star-x^1\|^2 \\
    \leq& \frac{L}{2}\|x^\star-y^1\|^2 - \frac{L}{2}\|x^\star-x^1\|^2 + \frac{L}{2}\|x^\star-x^1\|^2\\
    =& \frac{L}{2}\|x^\star-x^0\|^2,
    \end{split}
\end{equation*}
which implies \eqref{eq:result-tk}.

\subsection{Proof for Corollary \ref{cor:APBM}}\label{appen:cor}
Substituting \eqref{eq:tk_geq_k+1} into \eqref{eq:result-tk} yields that for all $k \geq 1$
\begin{equation}\label{eq:result-1}
    f(x^{k}) -f(x^\star) \leq \frac{2L\|x^\star-x^0\|^2}{(k+1)^2}.
\end{equation}
Next, we discuss the case of $f(x^0)-f(x^\star)$. By the smoothness of $f$, we obtain
\begin{equation}\label{eq:result-2}
    f(x^0) - f(x^\star) \leq \frac{L}{2}\|x^0-x^\star\|^2 \leq 2L\|x^0-x^\star\|^2.
\end{equation}
Combining \eqref{eq:result-2} with \eqref{eq:result-1} results in \eqref{eq:result}.

\end{appendix}


\bibliographystyle{ieeetran}
\bibliography{reference}

\end{document}